\documentclass[11pt,a4paper]{article}

\usepackage[T1]{fontenc}
\usepackage{lmodern}
\usepackage{microtype}
\usepackage[margin=1.05in]{geometry}
\usepackage{amsmath,amssymb,amsthm,mathtools}
\usepackage{enumitem}
\usepackage{xcolor}
\usepackage{hyperref}
\usepackage[nameinlink,noabbrev]{cleveref}

\hypersetup{
  colorlinks=true,
  linkcolor=blue!55!black,
  citecolor=blue!55!black,
  urlcolor=blue!55!black
}

\allowdisplaybreaks
\newtheorem{theorem}{Theorem}[section]
\newtheorem{lemma}[theorem]{Lemma}
\newtheorem{proposition}[theorem]{Proposition}
\newtheorem{corollary}[theorem]{Corollary}
\theoremstyle{definition}

\theoremstyle{remark}

\newcommand{\Homeop}{\operatorname{Homeo}_+}
\newcommand{\Supp}{\operatorname{Supp}}
\newcommand{\Cl}{\operatorname{Cl}}
\newcommand{\Split}{\operatorname{S}}
\newcommand{\BP}{\operatorname{BP}}

\newcommand{\Dyad}{\mathcal D}
\newcommand{\Z}{\mathbb Z}
\newcommand{\Q}{\mathbb Q}

\newcommand{\cP}{\mathcal P}
\newcommand{\cQ}{\mathcal Q}
\newcommand{\cR}{\mathcal R}
\newcommand{\preceqQI}{\preceq}
\newcommand{\asympQI}{\asymp}

\title{Every Copy of Thompson's Group $F$ in $F$ Is Undistorted}
\author{Gili Golan}
\date{}

\begin{document}
\maketitle

\begin{abstract}
Thompson's group $F$ is the group of all orientation-preserving piecewise-linear
homeomorphisms of the unit interval whose breakpoints are dyadic and whose slopes are
integer powers of $2$.  If
$H$ is a finitely generated subgroup of a finitely generated group $G$, its distortion
measures the difference between the intrinsic word metric of $H$ and the metric induced
from $G$; the subgroup is undistorted when these metrics are equivalent.  Guba and Sapir
asked whether $F$ contains a distorted copy of itself; the same question was also
suggested by Brin.  We answer this question negatively: every subgroup of $F$ isomorphic
to $F$ is undistorted in $F$.
\end{abstract}

\section{Introduction}

Thompson's group $F$ is the group of all orientation-preserving piecewise-linear
homeomorphisms of the unit interval whose breakpoints are dyadic rational numbers and
whose slopes are integer powers of $2$; see \cite{CFP1996}.

Let $\Omega$ be a set and let $\alpha,\beta:\Omega\to[0,\infty)$.  We write
$\alpha\preceqQI\beta$ if there is a constant $C\geq1$ such that
\[
        \alpha(x)\leq C\beta(x)+C
        \qquad(\forall x\in\Omega),
\]
and write $\alpha\asympQI\beta$ if both $\alpha\preceqQI\beta$ and
$\beta\preceqQI\alpha$.

Let $G$ be a finitely generated group. For a finite generating set $X$ of $G$ and
$g\in G$, let $|g|_X$ denote the word length of $g$ with respect to $X$. If $X$ and
$Y$ are two finite generating sets of $G$, then the elements of $X$ have uniformly
bounded $Y$-length, and hence
\[
        |g|_Y\leq C|g|_X
        \qquad(\forall g\in G)
\]
for some $C\geq1$. Interchanging $X$ and $Y$ gives the reverse inequality. Thus the
two word-length functions on $G$ satisfy
\[
        |g|_X\asympQI |g|_Y.
\]
Accordingly, when no particular finite generating set is specified, we write simply
$|g|_G$ for the word length of $g$ in $G$, understood up to $\asympQI$.

If $H\leq G$ is finitely generated, then $H$ is \emph{undistorted} in $G$ if its intrinsic word metric is equivalent to the metric induced from \(G\), that is, if
\[
        |h|_H\asympQI |h|_G.
\]
Equivalently, the inclusion $H\hookrightarrow G$ is a quasi-isometric embedding.
Otherwise $H$ is \emph{distorted} in $G$.

Distortion of subgroups of $F$ has been studied in several papers.  Every finitely
generated abelian subgroup of $F$, as well as the centralizer of every element of $F$,
is undistorted \cite{GubaSapir1999}.  Quasi-isometrically embedded copies of
$F^m\times\mathbb Z^n$ were constructed in \cite{Burillo1999,ClearyTaback2003}, and
Cleary proved that a natural copy of $\mathbb Z\wr\mathbb Z$ is undistorted in $F$
\cite{Cleary2006}.  The stabilizer of every finite set of rational points is undistorted
\cite{GolanSapirStabilizers2018}, and, more generally, every finitely generated closed
subgroup of $F$ is undistorted \cite{GolanPolakSapir2025}.

On the other hand, $F$ also contains distorted finitely generated subgroups.  Guba and
Sapir proved that, for every $d\geq2$, there is a finitely generated solvable subgroup of
$F$ whose distortion is at least $n^d$ \cite{GubaSapir1999}.  Moreover, Davis and
Olshanskii proved that $\mathbb Z\wr\mathbb Z$ contains finitely generated subgroups with
polynomial distortion of every degree \cite{DavisOlshanskii2011}; together with Cleary's
undistorted embedding of $\mathbb Z\wr\mathbb Z$ in $F$ \cite{Cleary2006}, this yields
subgroups of $F$ with polynomial distortion of every integer degree $d\geq2$.

In Sapir's survey \emph{Some group theory problems} \cite{Sapir2007},
Problem~6.10(a) asks whether $F$ contains a distorted copy of itself. The problem is
attributed to Guba and Sapir, and Sapir notes that the same question had also been
suggested by Brin.

Self-embeddings of $F$ were subsequently studied by Wassink \cite{Wassink2011}, partly
with this distortion question in mind.

Note that distorted copies of $F$ are known in other natural ambient groups. Wladis
proved that the natural embedding of $F$ into certain Thompson--Stein groups is
exponentially distorted \cite{Wladis2011}, while Burillo and Cleary proved that the
subgroup of inner automorphisms $\operatorname{Inn}(F)\cong F$ is at least
quadratically distorted in $\operatorname{Aut}(F)$ \cite{BurilloCleary2013}.

We answer \cite[Problem~6.10(a)]{Sapir2007} negatively.

\begin{theorem}\label{thm:main}
Let $\rho:F\to F$ be an injective homomorphism.  Then $\rho(F)$ is undistorted in $F$.
\end{theorem}

\paragraph{Acknowledgments.}
The author was partially supported by the Israel Science Foundation under grant
No.~2275/24.

\section{Preliminaries}

\subsection{Thompson's group \texorpdfstring{$F$}{F} and the groups
\texorpdfstring{$F_{2,r}$}{F(2,r)}}

For an interval $I$, let $\Homeop(I)$ denote the group of orientation-preserving
homeomorphisms of $I$, and let $\Dyad=\Z[1/2]$ denote the set of dyadic rational numbers.

Thompson's group $F$ is the group of all orientation-preserving piecewise-linear
homeomorphisms of $[0,1]$ with finitely many breakpoints in $\Dyad\cap(0,1)$ and slopes
which are integer powers of $2$.  On every affine piece, an element of $F$ has the form
\[
        t\longmapsto 2^n t+d,
        \qquad n\in\Z,\quad d\in\Dyad.
\]

We shall use the standard facts that the derived subgroup $F'$ is simple and that every
nontrivial normal subgroup of $F$ contains $F'$; see
\cite[Theorems~4.3 and~4.5]{CFP1996}.

Brown introduced the generalized Thompson groups $F_{n,r}$ in \cite{Brown1987}.  We
shall only need the binary groups $F_{2,r}$.  For an integer $r\geq1$, let $F_{2,r}$
denote the group of all orientation-preserving piecewise-linear homeomorphisms of
$[0,r]$ with finitely many breakpoints in $\Dyad$ and slopes which are integer powers
of $2$.  Thus $F_{2,1}=F$.  For every $r\geq1$, the group $F_{2,r}$ is isomorphic to
$F$; see \cite{Brown1987}.

\subsection{Dyadic subdivisions and binary forests}

A \emph{standard dyadic interval} in $[0,r]$, where $r\geq1$ is an integer, is an
interval
\[
        \left[\frac{k}{2^n},\frac{k+1}{2^n}\right]
        \subseteq[0,r]
\]
for integers $n,k\geq0$.  A \emph{dyadic subdivision} of $[0,r]$ is a finite
subdivision whose intervals are standard dyadic intervals.  Note that every dyadic
subdivision of $[0,r]$ refines
\[
        [0,1],\ [1,2],\ldots,[r-1,r].
\]

Let
\[
        \cP=(I_1,\ldots,I_m)
\]
be a dyadic subdivision of $[0,r]$, written from left to right.  Its associated finite
binary forest $T(\cP)$ has one vertex for every standard dyadic interval which is a
union of intervals of $\cP$.  The roots are $[0,1],\ldots,[r-1,r]$, the leaves are
$I_1,\ldots,I_m$, and an interval which is not a leaf is joined to its left and right
halves.  A non-leaf vertex together with its two children is called a \emph{caret}.
Starting with $r$ roots, every caret increases the number of leaves by one; hence a
forest with $m$ leaves has $m-r$ carets.

Let
\[
        \cP=(I_1,\ldots,I_m),
        \qquad
        \cQ=(J_1,\ldots,J_m)
\]
be two dyadic subdivisions of $[0,r]$ with the same number of intervals, written from
left to right.  They determine a unique increasing piecewise-linear homeomorphism of
$[0,r]$ which maps $I_i$ linearly onto $J_i$ for every $i$.  This homeomorphism belongs
to $F_{2,r}$, and the pair $(\cP,\cQ)$, or equivalently the pair of forests
$(T(\cP),T(\cQ))$, is called a \emph{forest-pair representation} of it.  Thus every
pair of finite binary forests with $r$ roots and the same number of leaves represents
an element of $F_{2,r}$.  Conversely, every element of $F_{2,r}$ admits a forest-pair
representation.  When $r=1$, it is also called a \emph{tree-pair representation}.

Two consecutive standard dyadic intervals of equal length are \emph{siblings} if their
union is a standard dyadic interval.  Let
\[
        \cP=(I_1,\ldots,I_m),
        \qquad
        \cQ=(J_1,\ldots,J_m)
\]
be a forest-pair representation.  It is \emph{reducible} if there exists $i$,
$1\leq i<m$, such that $I_i,I_{i+1}$ are siblings and $J_i,J_{i+1}$ are siblings.  In
this case, merging both sibling pairs gives another forest-pair representation of the
same element.  A representation is \emph{reduced} if it is not reducible.

Every reducible forest-pair representation can be transformed into a reduced
forest-pair representation of the same element by repeatedly performing such
reductions. Moreover, every element of $F_{2,r}$ has a unique reduced forest-pair
representation; see \cite{Brown1987}.

\subsection{Forest-pair size and word metrics}

For $g\in F_{2,r}$, let $\lambda_r(g)$ denote the number of leaves in either forest of
its reduced forest-pair representation, and put
\[
        \nu_r(g)=\lambda_r(g)-r.
\]
Thus $\nu_r(g)$ is the number of carets in either forest of the reduced representation.

The standard caret estimate for $F$ gives
\[
        |f|_F\asympQI \nu_1(f).
\tag{2.1}\label{eq:property-B}
\]
See \cite{BurilloClearyStein2001}.  We shall need the corresponding comparison for $F_{2,r}$ with
$r$ fixed.

\begin{lemma}\label{lem:fixed-roots}
For every fixed $r\geq1$,
\[
        |g|_{F_{2,r}}\asympQI \nu_r(g).
\tag{2.2}\label{eq:property-B-r}
\]
\end{lemma}

\begin{proof}
Choose a finite binary tree $R_r$ with $r$ leaves, and let
\[
        \eta_r:[0,1]\longrightarrow[0,r]
\]
be the increasing piecewise-linear homeomorphism which maps the leaf intervals of
$R_r$, from left to right, linearly onto $[0,1],\ldots,[r-1,r]$. Conjugation defines
an isomorphism
\[
        \iota_r:F_{2,r}\longrightarrow F,
        \qquad
        \iota_r(g)=\eta_r^{-1}\circ g\circ\eta_r.
\]
Let $c_r=r-1$, the number of carets of $R_r$, and let $g\in F_{2,r}$.

First, graft the $r$ roots of each forest in the reduced forest-pair representation
of $g$ to the $r$ leaves of $R_r$. The resulting tree pair represents
$\iota_r(g)$ and has $\nu_r(g)+c_r$ carets in each tree. Reduction can only decrease
the number of carets, so
\[
        \nu_1(\iota_r(g))\leq \nu_r(g)+c_r.
\]

Conversely, any finite binary tree can be expanded to contain $R_r$ as a top subtree
by adding at most $c_r$ carets: one takes the common refinement with $R_r$, and the
carets that must be added are among the $c_r$ carets of $R_r$. Start with the reduced
tree-pair representation of $\iota_r(g)$. First expand its domain tree to contain
$R_r$, performing the corresponding simultaneous expansions in the range tree; then
expand the range tree in the same way, again performing simultaneous expansions in
the domain tree. After at most $2c_r$ simultaneous expansions, both trees contain
$R_r$ as a common top subtree. Removing this subtree gives an $r$-root forest-pair
representation of $g$ with at most
\[
        \nu_1(\iota_r(g))+2c_r-c_r
        =\nu_1(\iota_r(g))+c_r
\]
carets in each forest. This representation is not necessarily reduced. Reducing it
if necessary can only decrease the number of carets. Hence
\[
        \nu_r(g)\leq \nu_1(\iota_r(g))+c_r.
\]

The two inequalities show that
\[
        \nu_1(\iota_r(g))\asympQI\nu_r(g).
\]
Transporting a finite generating set across the isomorphism $\iota_r$ shows that
\[
        |g|_{F_{2,r}}\asympQI|\iota_r(g)|_F;
\]
indeed, for corresponding generating sets the two word lengths are equal. Therefore,
by \eqref{eq:property-B},
\[
        |g|_{F_{2,r}}
        \asympQI|\iota_r(g)|_F
        \asympQI\nu_1(\iota_r(g))
        \asympQI\nu_r(g),
\]
as required.
\end{proof}

\subsection{Word lengths under homomorphisms}

Let $G$ and $Q$ be finitely generated groups, with finite generating sets $X$ and $Y$,
and let $\psi:G\to Q$ be a homomorphism.  Put
\[
        C_\psi=\max\{1,|\psi(x)|_Y:x\in X\}.
\]
Then
\[
        |\psi(g)|_Y\leq C_\psi |g|_X
        \qquad(\forall g\in G),
\tag{2.3}\label{eq:hom-length}
\]
since a word of length $n$ in the generators $X$ maps to a product of $n$ elements,
each of $Y$-length at most $C_\psi$.  In particular,
\[
        |\psi(g)|_Q\preceqQI |g|_G.
\]
Applied to an inclusion $H\leq G$, this gives
$|h|_G\preceqQI |h|_H$.  Thus, to prove that $H$ is undistorted in $G$, it is enough
to prove the reverse comparison.  Undistortion is transitive.

We shall use the following observation twice.

\begin{lemma}[Homomorphic detection of undistortion]\label{lem:detector}
Let $G$ and $Q$ be finitely generated groups, let $H\leq G$ be finitely generated, and
let $\psi:G\to Q$ be a homomorphism whose restriction to $H$ is an isomorphism onto
$\psi(H)$.  If $\psi(H)$ is undistorted in $Q$, then $H$ is undistorted in $G$.
\end{lemma}

\begin{proof}
Since $\psi|_H:H\to\psi(H)$ is an isomorphism,
\[
        |h|_H\asympQI |\psi(h)|_{\psi(H)}.
\]
Since $\psi(H)$ is undistorted in $Q$, and \eqref{eq:hom-length} gives
$|\psi(h)|_Q\preceqQI |h|_G$, we obtain
\[
        |h|_H
        \asympQI |\psi(h)|_{\psi(H)}
        \preceqQI |\psi(h)|_Q
        \preceqQI |h|_G.
\]

\end{proof}

\subsection{Supports and orbitals}

For $G\leq\Homeop([0,1])$, define
\[
        \Supp(G)=\{x\in(0,1):g(x)\neq x\text{ for some }g\in G\}.
\]
For a single element $g$, write $\Supp(g)=\Supp(\langle g\rangle)$.  A connected
component of $\Supp(G)$ is called an \emph{orbital} of $G$.  Every element of $F$ has
only finitely many orbitals.  Note that if $G=\langle g_1,\ldots,g_m\rangle$, then
\[
        \Supp(G)=\bigcup_{i=1}^m\Supp(g_i),
\]
so every finitely generated subgroup of $F$ has only finitely many orbitals.

\section{Split groups, closures, and finite denominators}

\subsection{Split groups}

Bleak introduced the split group of a subgroup of the group of piecewise-linear
homeomorphisms of an interval \cite{Bleak2008}.  We use the equivalent piecewise
formulation recorded by the author in \cite[Remark~11.11]{GolanGeneration}.

Let $I=[a,b]$ and let $G\leq\Homeop(I)$.  A map $f:I\to I$ is called a
\emph{piecewise-$G$ map} if there are points
\[
        a=t_0<t_1<\cdots<t_n=b
\]
and elements $g_1,\ldots,g_n\in G$ such that
\[
        f|_{[t_{i-1},t_i]}=g_i|_{[t_{i-1},t_i]}
        \qquad(1\leq i\leq n).
\]
The split group $\Split_I(G)$ is the set of all piecewise-$G$ maps.  Note that every
piecewise-$G$ map is automatically an orientation-preserving homeomorphism of $I$:
it is continuous, is strictly increasing on the pieces and hence on $I$, and fixes the two
endpoints.  Common refinements of the finitely many pieces show that $\Split_I(G)$ is
a subgroup of $\Homeop(I)$.  When the interval is clear, we write simply $\Split(G)$.

\subsection{Closure inside \texorpdfstring{$F$}{F}}

Closed subgroups of $F$ were introduced in \cite{GolanSapirSubgroups2017} using the
Stallings $2$-core of a subgroup of $F$, a construction defined by Guba and Sapir for
subgroups of diagram groups.
The Stallings $2$-core of a subgroup $L\leq F$ has a natural notion of accepting
elements of $F$.  The \emph{closure} of $L$, denoted $\Cl_F(L)$, is the subgroup of
$F$ consisting of all elements accepted by the core.  The core and closure constructions
for subgroups of $F$ were studied further in
\cite{GolanGeneration,GolanPolakMaximal2025,GolanPolakSapir2025}.

We shall not need the core itself here, but only the following equivalent description of
the closure.  For every subgroup $L\leq F$,
\[
        \Cl_F(L)=\Split(L)\cap F.
\tag{3.1}\label{eq:closure-split}
\]
Equivalently, $\Cl_F(L)$ consists precisely of those elements of $F$ which are piecewise
equal to elements of $L$; see \cite[Theorem~5.6]{GolanGeneration}.  A subgroup
$L\leq F$ is called \emph{closed} if
\[
        \Cl_F(L)=L.
\]

The intersection with $F$ in \eqref{eq:closure-split} is essential in general: the
split group $\Split(L)$ may contain elements with non-dyadic breakpoints, and therefore
need not itself be a subgroup of $F$.

We shall use the following theorem of the author and Sapir
\cite[Theorem~1.4]{GolanPolakSapir2025}.

\begin{theorem}\label{thm:closed-undistorted}
If $L\leq F$ is finitely generated, then $\Cl_F(L)$ is finitely generated and
undistorted in $F$.
\end{theorem}

\subsection{Breakpoints and the groups \texorpdfstring{$S_q(F)$}{Sq(F)}}

For a finite piecewise-linear homeomorphism $f$ of an interval, a point $x$ in the
interior of the interval is a \emph{breakpoint} if $f$ is not affine on any neighborhood
of $x$, equivalently if its left and right slopes at $x$ are different.  Let $\BP(f)$
denote the finite set of breakpoints.

\begin{lemma}\label{lem:rational-breakpoints}
If $f\in\Split(F)$, then $f$ is a finite piecewise-linear homeomorphism, each affine
piece has the form
\[
        t\longmapsto 2^n t+d,
        \qquad n\in\Z,\quad d\in\Dyad,
\]
and every breakpoint of $f$ is rational.
\end{lemma}

\begin{proof}
Choose a piecewise-$F$ presentation and refine it at the breakpoints of the finitely many
local elements.  This gives the asserted affine formulas.  Let $x$ be a breakpoint.  On sufficiently small left and right neighborhoods of $x$,
the two affine formulas have the form
\[
        2^m t+c
        \qquad\text{and}\qquad
        2^n t+d,
        \qquad c,d\in\Dyad.
\]
Continuity at $x$ gives
\[
        2^m x+c=2^n x+d.
\]
If $m=n$, then $c=d$, so the two affine formulas coincide, contrary to the
definition of a breakpoint.  Hence $m\neq n$ and
\[
        x=\frac{d-c}{2^m-2^n}\in\Q.
\]
\end{proof}

For a positive odd integer $q$, put
\[
        A_q=[0,1]\cap\left(\frac1q\Dyad\right)
\]
and define
\[
        S_q(F)=\{f\in\Split(F):\BP(f)\subseteq A_q\}.
\tag{3.2}\label{eq:Sq-def}
\]
Note that $f\in\Split(F)$ belongs to $S_q(F)$ if and only if it admits a piecewise-$F$
presentation whose splitting points all lie in $A_q$.

\begin{proposition}\label{prop:Sq}
For every positive odd integer $q$, the set $S_q(F)$ is a subgroup of $\Split(F)$, and
\[
        \Split(F)=\bigcup_{q\text{ odd}}S_q(F).
\tag{3.3}\label{eq:Sq-union}
\]
\end{proposition}

\begin{proof}
Every $f\in\Split(F)$ maps $A_q$ onto itself.  Indeed, let
\[
        x=\frac{m}{q2^k}\in A_q.
\]
Choose an affine piece of $f$ containing $x$.  On this piece,
\[
        f(t)=2^n t+d,
        \qquad d\in\Dyad,
\]
and hence
\[
        f(x)=2^n\frac{m}{q2^k}+d\in A_q.
\]
Applying the same argument to $f^{-1}$ gives $f(A_q)=A_q$.

If $f,g\in S_q(F)$, then
\[
        \BP(f\circ g)\subseteq \BP(g)\cup g^{-1}(\BP(f))\subseteq A_q,
\]
and
\[
        \BP(f^{-1})=f(\BP(f))\subseteq A_q.
\]
Thus $S_q(F)$ is a subgroup.

Finally, let $f\in\Split(F)$. By Lemma~\ref{lem:rational-breakpoints}, $f$ has finitely
many rational breakpoints. Choose a positive odd integer $q$ divisible by the odd parts
of the reduced denominators of all breakpoints of $f$. Then every breakpoint of $f$
belongs to $A_q$, so $f\in S_q(F)$. This proves \eqref{eq:Sq-union}.
\end{proof}

\subsection{Scaling}

Let
\[
        d_q:[0,1]\longrightarrow[0,q],
        \qquad d_q(t)=qt.
\]
Conjugation by $d_q$ defines an injective homomorphism
\[
        \sigma_q:S_q(F)\longrightarrow F_{2,q},
        \qquad
        \sigma_q(f)=d_q\circ f\circ d_q^{-1}.
\tag{3.4}\label{eq:sigma}
\]
Indeed, because $d_q$ is affine, conjugation simply stretches the breakpoint set, so
\[
        \BP(\sigma_q(f))=q\BP(f)\subseteq\Dyad,
\]
and it preserves all slopes.  Explicitly,
\[
        \sigma_q(f)(t)=qf(t/q).
\tag{3.5}\label{eq:sigma-formula}
\]

\section{The scaled copy of \texorpdfstring{$F$}{F} is undistorted}

Fix a positive odd integer $q$. We prove that the restriction
\[
        \sigma_q|_F:F\longrightarrow F_{2,q}
\]
is a quasi-isometric embedding.

\subsection{The canonical dyadic subdivision}

Let $h\in F_{2,r}$. A standard dyadic interval $I\subseteq[0,r]$ is
\emph{admissible for $h$} if $h$ is affine on $I$ and $h(I)$ is a standard dyadic
interval. Admissibility passes to every standard dyadic subinterval of an admissible
interval. Indeed, every such subinterval belongs to the subdivision of $I$ into
$2^k$ equal intervals for some $k\geq0$, and the affine map $h|_I$ sends this
subdivision onto the corresponding subdivision of $h(I)$.

Starting with the $r$ root intervals $[0,1],\ldots,[r-1,r]$, recursively bisect each
interval that is not admissible. This process terminates. Indeed, let
$(\cP,\cQ)$ be any forest-pair representation of $h$. Every interval of $\cP$,
equivalently every leaf of $T(\cP)$, is admissible. The construction begins at the
roots of $T(\cP)$; whenever it bisects a vertex of $T(\cP)$, that vertex is not a leaf
and its two halves are its children in $T(\cP)$. Thus every interval bisected by the
construction is a non-leaf vertex of the finite forest $T(\cP)$. Hence only finitely
many bisections occur.

Let $\cR=(I_1,\ldots,I_m)$ be the terminal subdivision produced by the process. Since
each $I_i$ is admissible, the intervals $h(I_i)$ are standard dyadic intervals. Because
$h$ is increasing, they occur in the same order, have disjoint interiors, and cover
$[0,r]$. Thus
\[
        (\cR,h(\cR))
\]
defines a forest-pair representation of $h$. It is reduced: if two consecutive sibling
intervals $I_i,I_{i+1}$ and their image intervals could be merged, then their common
parent would itself be admissible and would not have been bisected. Hence $\cR$ is the
domain subdivision of the reduced forest-pair representation of $h$.

Consequently, $\nu_r(h)$ is exactly the number of intervals bisected in this recursive
construction.

\subsection{Admissibility under dilation}

Let
\[
        I=[a,a+\ell]\subseteq[0,1]
\]
be a standard dyadic interval. Thus
\[
        \ell=2^{-n}
\]
for some $n\geq0$, and $a/\ell\in\Z$. The stretched interval $qI$ is the union of the
$q$ consecutive standard dyadic intervals
\[
        I^{(j)}=[qa+(j-1)\ell,\,qa+j\ell],
        \qquad 1\leq j\leq q.
\tag{4.1}\label{eq:copies}
\]
We call them the $q$ copies of $I$.

We shall use two elementary observations. First, if $L=2^{-m}$ for some $m\geq0$,
then an interval
\[
        [b,b+L]
\]
is a standard dyadic interval if and only if
\[
        \frac bL\in\Z,
\]
since the standard dyadic intervals of length $L$ are exactly the intervals
$[kL,(k+1)L]$, $k\in\Z$. Second, the points separating the $q$ copies of $I$ have
preimages under $d_q$ equal to
\[
        a+\frac jq\ell,
        \qquad 1\leq j<q,
\]
and none of these points is dyadic. Indeed, since $a/\ell\in\Z$ and $q$ is odd,
after division by $\ell$ their fractional parts have a nontrivial odd denominator.

\begin{lemma}\label{lem:admissible-copies}
Let $f\in F$ and put $g=\sigma_q(f)$. A standard dyadic interval $I\subseteq[0,1]$
is admissible for $f$ if and only if all its $q$ copies are admissible for $g$.
\end{lemma}

\begin{proof}
Suppose first that $I$ is admissible for $f$, and write
\[
        f(I)=[b,b+L].
\]
Then $L=2^{-m}$ for some $m\geq0$ and $b/L\in\Z$. The map
$g=d_qfd_q^{-1}$ is affine on $qI$, and
\[
        g(I^{(j)})=[qb+(j-1)L,\,qb+jL].
\tag{4.2}\label{eq:image-copy}
\]
The left endpoint of this interval divided by its length is
\[
        q\frac bL+j-1\in\Z,
\]
so \eqref{eq:image-copy} is a standard dyadic interval. Hence every copy of $I$ is
admissible for $g$.

Conversely, suppose that all $q$ copies of $I$ are admissible for $g$. If $f$ were not
affine on $I$, then $I$ would contain a breakpoint $x$ of $f$ in its interior. Since
$x$ is dyadic, the observation above shows that $qx$ cannot be a boundary between two
copies of $I$. Hence $qx$ lies in the interior of one of the copies. But $qx$ is a
breakpoint of $g$, contradicting the admissibility of that copy. Thus $f$ is affine on
$I$.

Write
\[
        f(I)=[b,b+L].
\]
Since the left endpoint of $I$ is dyadic and $f\in F$, its image $b$ under $f$ is
dyadic; and since both $|I|$ and the slope of $f|_I$ are powers of $2$, we have
$L=2^{-m}$ for some $m\geq0$.
Hence $b/L$ is dyadic. The first copy is admissible, and
\[
        g(I^{(1)})=[qb,qb+L],
\]
so the characterization above gives
\[
        q\frac bL\in\Z.
\]
Since $q$ is odd and $b/L$ is dyadic, this implies $b/L\in\Z$. Therefore
$f(I)=[b,b+L]$ is a standard dyadic interval, and $I$ is admissible for $f$.
\end{proof}

\begin{proposition}\label{prop:scaled-undistorted}
For every positive odd integer $q$, the restriction
\[
        \sigma_q|_F:F\longrightarrow F_{2,q}
\]
is a quasi-isometric embedding.
\end{proposition}

\begin{proof}
Let $f\in F$ and put $g=\sigma_q(f)$. For every interval $I$ bisected in the canonical
construction for $f$, Lemma~\ref{lem:admissible-copies} gives a copy of $I$ which is
not admissible for $g$; choose one such copy and denote it by $J_I$.

Each $J_I$ is bisected in the canonical construction for $g$. Indeed, since
admissibility passes to standard dyadic subintervals, all the dyadic ancestors of the
nonadmissible interval $J_I$ are also nonadmissible. Thus the construction reaches and
bisects $J_I$.

The assignment $I\mapsto J_I$ is injective. If two source intervals have different
lengths, then so do their chosen copies. Distinct source intervals of the same length
have disjoint interiors, and so do their stretched blocks. Since each chosen copy is
contained in the corresponding stretched block, the chosen copies are distinct.
Consequently,
\[
        \nu_1(f)\leq \nu_q(g).
\tag{4.3}\label{eq:nu-lower}
\]

For the reverse inequality, let the reduced tree-pair representation of $f$ have
$\lambda_1(f)$ leaves. Each of its leaf intervals is admissible, so by
Lemma~\ref{lem:admissible-copies} its $q$ copies are admissible for $g$. These copies give a
forest-pair representation of $g$ with $q\lambda_1(f)$ leaves, hence with
\[
        q\lambda_1(f)-q=q\nu_1(f)
\]
carets. Reduction can only decrease the number of carets, so
\[
        \nu_q(g)\leq q\nu_1(f).
\tag{4.4}\label{eq:nu-upper}
\]
Combining \eqref{eq:nu-lower} and \eqref{eq:nu-upper},
\[
        \nu_1(f)\leq\nu_q(\sigma_q(f))\leq q\nu_1(f).
\tag{4.5}\label{eq:nu-comparison}
\]
By \eqref{eq:property-B} and \eqref{eq:property-B-r}, this is precisely the required
quasi-isometric comparison of word lengths.
\end{proof}

\begin{corollary}\label{cor:F-undistorted-L}
Let $L$ be a finitely generated group satisfying
\[
        F\leq L\leq S_q(F).
\]
Then $F$ is undistorted in $L$.
\end{corollary}

\begin{proof}
Apply Lemma~\ref{lem:detector} to the injective homomorphism
\[
        \sigma_q|_L:L\longrightarrow F_{2,q}.
\]
Its restriction to $F$ is an isomorphism onto the undistorted subgroup
$\sigma_q(F)$ by Proposition~\ref{prop:scaled-undistorted}.
Hence the lemma applies.
\end{proof}

\section{Semiconjugacies and split groups}

We use the notion of semiconjugacy for actions on an open interval employed by
Brum--Matte Bon--Rivas--Triestino \cite[Section~2.1]{BMRT2026}. Let
$I=(a,b)$ and $J=(c,d)$ be nonempty open intervals. A monotone map $p:I\to J$ is called \emph{proper} if
\[
        \inf p(I)=c
        \qquad\text{and}\qquad
        \sup p(I)=d.
\tag{5.1}\label{eq:proper}
\]
Here monotone means either nondecreasing or nonincreasing. If a group $G$ acts on
$I$ and $J$ by
\[
        \alpha:G\longrightarrow\Homeop(I),
        \qquad
        \beta:G\longrightarrow\Homeop(J),
\]
a \emph{semiconjugacy} from $\alpha$ to $\beta$ is a proper monotone map
$p:I\to J$ satisfying
\[
        p\circ\alpha(g)=\beta(g)\circ p
        \qquad(\forall g\in G).
\tag{5.2}\label{eq:semiconjugacy-def}
\]
Semiconjugacy in this sense is an equivalence relation; see Kim--Koberda--Mj
\cite[Lemma~2.2]{KKM2019} and \cite[Remark~2.1.4]{BMRT2026}. We shall also use the
observation that if the target action $\beta$ is minimal, then any semiconjugacy from
$\alpha$ to $\beta$ is continuous \cite[Remark~2.1.13]{BMRT2026}.

The following proposition is the mechanism by which the semiconjugacy will
be used.

\begin{proposition}[Extension to split groups]\label{prop:split-extension}
Let $I$ and $J$ be compact intervals, let
\[
        H\leq\Homeop(I),
        \qquad
        G\leq\Homeop(J),
\]
and let $\phi:H\to G$ be an isomorphism. Suppose that
$p:I\to J$ is a continuous surjective monotone map satisfying
\[
        p(h(t))=\phi(h)(p(t))
        \qquad(\forall h\in H,\ \forall t\in I).
\tag{5.3}\label{eq:semi-phi}
\]
Then for every $k\in\Split_I(H)$ the formula
\[
        \Phi(k)(y)=p(k(t)),
        \qquad t\in p^{-1}(y),
\tag{5.4}\label{eq:Phi-def}
\]
is well defined and defines an element $\Phi(k)\in\Split_J(G)$. The resulting map
\[
        \Phi:\Split_I(H)\longrightarrow\Split_J(G)
\]
is a homomorphism and $\Phi|_H=\phi$.
\end{proposition}

\begin{proof}
Fix $k\in\Split_I(H)$. Choose a piecewise-$H$ presentation
\[
        I=[c_0,c_1]\cup\cdots\cup[c_{n-1},c_n],
        \qquad c_0<\cdots<c_n,
\]
and elements $h_1,\ldots,h_n\in H$ such that
\[
        k|_{[c_{i-1},c_i]}=h_i|_{[c_{i-1},c_i]}.
\tag{5.5}\label{eq:k-pieces}
\]
We first prove that \eqref{eq:Phi-def} is well defined.

Let $s,t\in I$ satisfy
\[
        p(s)=p(t)=y.
\]
Assume $s\leq t$. Since $p$ is monotone, it is constant on $[s,t]$: for every
$u\in[s,t]$, monotonicity puts $p(u)$ between the equal values $p(s)$ and $p(t)$.
Thus
\[
        p(u)=y\qquad(\forall u\in[s,t]).
\tag{5.6}\label{eq:fiber-interval}
\]
On every nonempty intersection $[s,t]\cap[c_{i-1},c_i]$, equations
\eqref{eq:semi-phi}, \eqref{eq:k-pieces}, and \eqref{eq:fiber-interval} give
\[
        p(k(u))=p(h_i(u))=\phi(h_i)(p(u))=\phi(h_i)(y).
\tag{5.7}\label{eq:constant-piece}
\]
Hence $p\circ k$ is constant on each such intersection. If $[s,t]$ crosses the
splitting point $c_i$, then $p(c_i)=y$ and
\[
        h_i(c_i)=k(c_i)=h_{i+1}(c_i).
\]
Therefore
\[
        \phi(h_i)(y)
        =p(h_i(c_i))
        =p(h_{i+1}(c_i))
        =\phi(h_{i+1})(y).
\tag{5.8}\label{eq:constants-match}
\]
The constants in \eqref{eq:constant-piece} consequently agree across every splitting
point between $s$ and $t$. Thus $p\circ k$ is constant on $[s,t]$, and in particular
\[
        p(k(s))=p(k(t)).
\]
This proves that \eqref{eq:Phi-def} is independent of the choice of
$t\in p^{-1}(y)$. Since $p$ is surjective, it therefore defines a map
$\Phi(k):J\to J$, characterized by
\[
        \Phi(k)\circ p=p\circ k.
\tag{5.9}\label{eq:Phi-intertwine}
\]

We next show that $\Phi(k)$ is piecewise $G$. For each $i$, let
\[
        E_i=p([c_{i-1},c_i]).
\]
Because $p$ is continuous and monotone, $E_i$ is the closed interval with endpoints
$p(c_{i-1})$ and $p(c_i)$. The nondegenerate $E_i$ cover $J$ and, when listed from
left to right, form a finite subdivision of $J$; if $p$ is nonincreasing, their order is
simply reversed. For $y\in E_i$, choose $t\in[c_{i-1},c_i]$ with $p(t)=y$. Then
\[
        \Phi(k)(y)=p(k(t))=p(h_i(t))=\phi(h_i)(y).
\]
Thus
\[
        \Phi(k)|_{E_i}=\phi(h_i)|_{E_i}.
\tag{5.10}\label{eq:Phi-local}
\]
After deleting the degenerate intervals, these restrictions give a finite piecewise-$G$
presentation of $\Phi(k)$. Hence $\Phi(k)\in\Split_J(G)$.

For $k,\ell\in\Split_I(H)$, equation \eqref{eq:Phi-intertwine} gives
\[
\begin{aligned}
        \Phi(k\circ\ell)\circ p
        &=p\circ k\circ\ell\\
        &=\Phi(k)\circ p\circ\ell\\
        &=\Phi(k)\circ\Phi(\ell)\circ p.
\end{aligned}
\]
Since $p$ is surjective,
\[
        \Phi(k\circ\ell)=\Phi(k)\circ\Phi(\ell).
\]
Thus $\Phi$ is a homomorphism. Finally, if $h\in H$, then
\[
        \Phi(h)\circ p=p\circ h=\phi(h)\circ p,
\]
and surjectivity of $p$ gives $\Phi(h)=\phi(h)$. Hence $\Phi|_H=\phi$.
\end{proof}

\section{A faithful orbital and the semiconjugacy}

Fix an injective homomorphism
\[
        \rho:F\longrightarrow F,
\]
and put
\[
        H=\rho(F),
        \qquad
        K=\Cl_F(H).
\tag{6.1}\label{eq:H-K}
\]

\subsection{A faithful orbital}

\begin{lemma}\label{lem:faithful-orbital}
There is an orbital $J=(a,b)$ of $H$ on which the restricted action of $H$ is faithful.
\end{lemma}

\begin{proof}
The group $H$ is finitely generated, so it has finitely many orbitals
$J_1<\cdots<J_m$. Each orbital is $H$-invariant. For each $i$, define
\[
        \rho_i:F\longrightarrow\Homeop(J_i),
        \qquad
        \rho_i(f)=\rho(f)|_{J_i},
\]
and put $N_i=\ker\rho_i$. Then $N_i\trianglelefteq F$ and
\[
        \bigcap_{i=1}^mN_i=\{1\},
\]
because an element in the intersection has image under $\rho$ equal to the identity on
$\Supp(H)$ and also outside $\Supp(H)$. If every $N_i$ were nontrivial, then each
would contain $F'$, a contradiction. Thus some $N_i$ is trivial.
\end{proof}

Fix such an orbital $J=(a,b)$ and put
\[
        H_J=\{h|_J:h\in H\}\leq\Homeop(J).
\]
The restriction map $H\to H_J$ is an isomorphism, by the faithfulness of the action on
$J$. The action of $H_J$ has no global fixed point in $J$. Since every element of $H$
fixes $a$ and $b$, we identify each element of $H_J$ with its unique endpoint-fixing
extension to $[a,b]$.

\subsection{The locally moving theorem}

An action on an open interval is \emph{locally moving} if, for every nonempty open
subinterval $U$, the subgroup of elements supported in $U$ has no global fixed point in
$U$. It is \emph{irreducible} if it has no global fixed point.

We use \cite[Corollary~4.1.2]{BMRT2026} in the following form.

\begin{theorem}[Brum--Matte Bon--Rivas--Triestino]\label{thm:BMRT}
Let $X$ be an open interval, let $G\leq\Homeop(X)$ be locally moving, let $Y$ be an open interval,
and let
\[
        \psi:G\longrightarrow\Homeop(Y)
\]
be a faithful irreducible action. If the support of some nontrivial element in the
$\psi$-action is bounded away from one end of $Y$, then the action on $Y$ is semiconjugate
to the action on $X$.
\end{theorem}

The standard action of $F$ on $(0,1)$ is locally moving and minimal; see
\cite{CFP1996,BMRT2026}.

\begin{lemma}\label{lem:bounded-support}
The action of $H_J$ contains a nontrivial element whose support is bounded away from
both endpoints of $J$.
\end{lemma}

\begin{proof}
Choose $1\neq f\in F'$. Faithfulness implies that $\rho(f)|_J$ is nontrivial, and
$\rho(f)\in H'$. Since every element of $H$ fixes $a$ and $b$, the chain rule shows
that the endpoint slope characters
\[
        \chi_a^+(h)=\log_2 h'(a^+),
        \qquad
        \chi_b^-(h)=\log_2 h'(b^-)
\]
vanish on $H'$. Hence $\rho(f)|_J$ has slope $1$ near both endpoints of $J$. Since it
fixes the endpoints, it is the identity near each endpoint of $J$.
\end{proof}

Apply Theorem~\ref{thm:BMRT} to the standard action of $F$ on $(0,1)$ and the faithful action
on $J$. By symmetry of semiconjugacy, choose a proper monotone semiconjugacy
\[
        p_0:J\longrightarrow(0,1)
\]
in this direction, so that
\[
        p_0(\rho(f)(t))=f(p_0(t))
        \qquad(\forall f\in F,\ \forall t\in J).
\tag{6.2}\label{eq:orbital-semi-open}
\]
Since the target action is minimal, $p_0$ is continuous by
\cite[Remark~2.1.13]{BMRT2026}. Properness means
\[
        \inf p_0(J)=0,
        \qquad
        \sup p_0(J)=1,
\]
and a continuous monotone map with these endpoint extrema is onto $(0,1)$. It also
follows that $p_0$ has limits at the endpoints of $J$, equal to $0$ and $1$ in one
order or the other. Hence $p_0$ extends uniquely to a continuous surjective monotone map
\[
        p:[a,b]\longrightarrow[0,1].
\tag{6.3}\label{eq:p-closed}
\]
Define
\[
        \phi_J:H_J\longrightarrow F,
        \qquad
        \phi_J(h|_J)=\rho^{-1}(h)
        \qquad(\forall h\in H).
\tag{6.4}\label{eq:phiJ}
\]
Since both $\rho:F\to H$ and the restriction map $H\to H_J$ are isomorphisms,
$\phi_J:H_J\to F$ is an isomorphism. Equation
\eqref{eq:orbital-semi-open}, together with continuity at the endpoints, becomes
\[
        p(g(t))=\phi_J(g)(p(t))
        \qquad(\forall g\in H_J,\ \forall t\in[a,b]).
\tag{6.5}\label{eq:orbital-semi}
\]

\section{Proof of the main theorem}

Since $H\cong F$ is finitely generated, Theorem~\ref{thm:closed-undistorted} implies
that $K=\Cl_F(H)$ is finitely generated.

Since $K\subseteq\Split(H)$, every element of $K$ is piecewise $H$. Every element of
$H$ fixes $a$ and $b$, so every element of $K$ preserves $[a,b]$. Restriction therefore
defines a homomorphism
\[
        \operatorname{res}_J:K\longrightarrow\Split_{[a,b]}(H_J).
\]
The restriction of $\operatorname{res}_J$ to $H$ is the isomorphism $H\to H_J$
provided by the faithful action on $J$. We now apply
Proposition~\ref{prop:split-extension} to the isomorphism
$\phi_J:H_J\to F$ and the continuous surjective monotone map
$p:[a,b]\to[0,1]$. The equivariance relation
\eqref{eq:orbital-semi} holds, so the hypotheses of
Proposition~\ref{prop:split-extension} are satisfied. Hence the proposition gives a
homomorphism
\[
        \Phi:\Split_{[a,b]}(H_J)\longrightarrow\Split(F)
\]
whose restriction to $H_J$ is $\phi_J$. Put
\[
        \beta=\Phi\circ\operatorname{res}_J:K\longrightarrow\Split(F)
\]
and let
\[
        L=\beta(K).
\]
Since $K$ is finitely generated, so is $L=\beta(K)$, and
\[
        F=\beta(H)\leq L\leq\Split(F),
\]
while $\beta|_H:H\to F$ is an isomorphism.

Choose a finite generating set of $L$. By \eqref{eq:Sq-union}, there is a positive odd
integer $q$ such that all of its generators belong to $S_q(F)$. Since $S_q(F)$ is a
subgroup,
\[
        F\leq L\leq S_q(F).
\]
By Corollary~\ref{cor:F-undistorted-L}, the subgroup $F$ is undistorted in $L$. Since
$\beta|_H:H\to F$ is an isomorphism and $F$ is undistorted in $L$,
Lemma~\ref{lem:detector}, applied to
\[
        \beta:K\longrightarrow L,
\]
shows that $H$ is undistorted in $K$.

Finally, $K=\Cl_F(H)$ is undistorted in the ambient group $F$ by
Theorem~\ref{thm:closed-undistorted}. Since undistortion is transitive, $H=\rho(F)$ is
undistorted in $F$. This proves Theorem~\ref{thm:main}.

\bigskip
\noindent
\textsc{Department of Mathematics, Ben-Gurion University of the Negev,
Be'er Sheva, Israel}\\
\textit{Email address:}
\href{mailto:golangi@bgu.ac.il}{\texttt{golangi@bgu.ac.il}}

\end{document}